\documentclass[11pt]{amsart}
\usepackage{latexsym}
\usepackage{amssymb,amsmath,mathabx, amscd}
\usepackage[pdftex]{graphicx}
\usepackage{enumerate}
\usepackage{colonequals, stmaryrd}
\usepackage{tikz}
\usepackage{tikz-cd}
\usepackage[margin=.95in]{geometry}
\usepackage{hyperref}

\usepackage{hyperref, mathrsfs}

\renewcommand{\O}{\mathcal{O}}

\newcommand{\qq}{\mathbb{Q}}
\newcommand{\zz}{\mathbb{Z}}

\renewcommand{\O}{\mathcal{O}}

\newcommand{\RaG}{R\langle G \rangle}
\newcommand{\FaG}{F\langle G \rangle}

\newcommand{\fHom}{\mathscr{H}\hspace{-2pt}om}

\newcommand{\Res}{\text{Res}}

\newcommand{\End}{\operatorname{End}}

\newcommand{\Frac}{\operatorname{Frac}}

\newcommand{\Gal}{\operatorname{Gal}}

\newcommand{\Hom}{\operatorname{Hom}}

\newcommand{\rk}{\operatorname{rk}}
\newcommand{\opp}{\operatorname{opp}}
\newcommand{\Char}{\operatorname{char}}
\newcommand{\Spec}{\operatorname{Spec}}

\newcommand{\isog}{\operatorname{isog}}

\newtheorem{thm}{Theorem}[section]
\newtheorem{ithm}{Theorem}
\newtheorem{lem}[thm]{Lemma}

\newtheorem{prop}[thm]{Proposition}

\newtheorem{icor}[ithm]{Corollary}

\theoremstyle{definition}

\newtheorem{example}[thm]{Example}

\theoremstyle{remark}

\title[Abelian varieties isogenous to a power of an elliptic curve]{Abelian varieties isogenous to a power of an \\  elliptic curve over a Galois extension}
\date{\today}
\author{Isabel Vogt}
\address{Department of Mathematics, Massachusetts Institute of Technology, Cambridge, MA 02139}
\email{ivogt@mit.edu}

\begin{document}

\maketitle

\begin{abstract}
Given an elliptic curve $E/k$ and a Galois extension $k'/k$, we construct an exact functor from torsion-free modules over the endomorphism ring $\End E_{k'}$ with a semilinear $\Gal(k'/k)$ action to abelian varieties over $k$ that are $k'$-isogenous to a power of $E$.  As an application, we give a simple proof that every elliptic curve with complex multiplication geometrically is isogenous over the ground field to one with complex multiplication by a maximal order.
\end{abstract}

\section{Introduction}

Let $E$ be an elliptic curve over a field $k$.  As in \cite{bjorn}, the theory of abelian varieties isogenous over $k$ to a power of $E$ is related to the theory of finitely presented torsion-free modules over the endomorphism ring $R_k \colonequals \End_k{E}$.  To recall briefly, there is a functor
\[\fHom_{R_k}(-,E) \colon \left\{ \text{finitely presented} \atop  \text{left $R_k$-modules} \right\}^{\opp} \to \left\{ \text{commutative proper} \atop  \text{$k$-group schemes} \right\},\]
such that for $M$ a finitely presented $R_k$-module and $C$ a $k$-scheme, we have
\[\Hom_k(C, \fHom_{R_k}(M, E)) = \Hom_{R_k}(M, \Hom_k(C, E)). \]
Restricting to torsion-free modules, we obtain a functor
\[\fHom_{R_k}(-,E) \colon \left\{ \text{fin. pres. tors. free} \atop  \text{left $R_k$-modules} \right\}^{\opp} \to \left\{ \text{abelian $k$-varieties} \atop  \text{isogenous to $E^r$, $r\in \zz$} \right\}.\]
This functor is fully faithful, but not in general an equivalence of categories.  For example, if $\Char k =0$ and $E$ does not have complex multiplication, then the image of this functor consists entirely of the powers of $E$; yet it is possible for $E$ to be $k$-isogenous to a non-isomorphic curve.

In fact, this functor is never surjective if $E/k$ acquires complex multiplication (CM) over a separable quadratic extension $k'/k$, as we now show.  Suppose that $E_{k'}$ has CM by the order $\O$.  Let $E'$ be the $k'/k$ quadratic twist of $E$.  Let $\alpha \in \O$ be a purely imaginary element of $\O$, so that the nontrivial element $\sigma \in \Gal(k'/k)$ acts as ${}^\sigma \alpha = -\alpha$.  Multiplication by $\alpha$ composed with the isomorphism $E_{k'} \xrightarrow{\simeq} E'_{k'}$ is Galois stable, and so descends to an isogeny $E \to E'$ over $k$.

In this note we describe a generalization of the functor $\fHom_{R_k}(-,E)$, which in particular addresses the case when $\End_{k}E \neq \End_{\bar{k}}E$.   We will show that the essential image of this functor always contains all of the quadratic twists of $E$.  

More generally, one may also consider abelian varieties, as in the case of nontrivial twists, that are isogenous to a power of an elliptic curve only after passing to a Galois extension of the ground field.  When restricted to torsion-free modules, the image of the functor we construct will lie in the category of abelian varieties that become isogenous to a power of $E$ over a Galois extension.  This functor may therefore shed light upon abelian varieties that are not isogenous to a power of $E$ over the ground field, and therefore missed by the previous functor, but become isogenous to a power of $E$ after making a suitable extension.

Let $k'/k$ be a finite Galois extension and let $G \colonequals \Gal(k'/k)$.  As every proper commutative group scheme over a field is projective, the category of commutative proper $k$-group schemes is equivalent to the category of commutative proper $k'$-group schemes equipped with descent data for $k'/k$.  For this reason we may identify commutative proper $k$-group schemes with commutative proper $k'$-group schemes with an action of $G$, such that all maps, including the structure maps, are $G$-equivariant. 

Let $R = R_{k'} \colonequals \End E_{k'}$ be the endomorphism ring of the base change $E_{k'}$ of $E$ to $k'$.  In particular, $R$ is noetherian.  This inherits an action of $G$.  We denote this action by $r \mapsto {}^\sigma r$ for $\sigma \in G$ and $r \in R$.  We may then form the twisted group ring $\RaG$ as the free $R$-module
\[\RaG = \bigoplus_{\sigma \in G} R \cdot \sigma \]
with the commutation relation $\sigma r = {}^\sigma r \sigma$.  In this way, a module over $\RaG$ is an $R$-module with a semilinear $G$-action.

We give a construction of the following in Section \ref{cattheory}. 
\begin{ithm}
There exists an exact functor
\[\fHom_{\RaG}(-, E_{k'}) \colon \left\{ \text{fin. pres. left} \atop  \text{$\RaG$-modules} \right\}^{\opp} \to \left\{ \text{commutative proper} \atop  \text{$k$-group schemes} \right\},\]
such that
\begin{enumerate}
\item For a finitely presented $\RaG$-module $M$ and any $k$-scheme $C$, we have
\[\Hom_{k}(C, \fHom_{\RaG}(M, E_{k'})) = \Hom_{\RaG}(M, \Hom_{k'}(C_{k'}, E_{k'})). \]
\item The functor agrees with $\fHom_R(-,E_{k'})$ under the forgetful functors mapping $\RaG$-modules to $R$-modules and base-changing $k$-schemes to $k'$-schemes.  
\end{enumerate}
\end{ithm}

In particular, this shows that if $M$ is an $\RaG$-module that is torsion-free as an $R$-module, then $\fHom_{\RaG}(M, E_{k'})$ is an abelian variety $A$ over $k$, such that $A_{k'}$ is isogenous to a power of $E$.  However, we may say something about the $k$-isogeny class of $A$ as well.  Recall that the category of abelian varieties over $k$ up to isogeny has the same objects as the category of abelian varieties over $k$, but all isogenies are inverted.

Let $F = R \otimes_\zz \qq$.  Since $F$ is semisimple, the category of finitely presented modules over the twisted group algebra $\FaG$ is semisimple (as in the case when the action of $G$ on $F$ is trivial, see Lemma \ref{semisimple} below).  Let $S_1, \cdots, S_\ell$ denote the simple objects.

\begin{ithm}\label{isog}
There is a functor
\[\fHom_{\FaG}(-, E_{k'}) \colon \left\{\text{ fin. pres.} \atop  \text{$\FaG$-modules} \right\}^{\opp} \to \left\{ \text{abelian varieties over $k$} \atop \text{up to isogeny} \right\},\]
compatible with the functor $\fHom_{\RaG}(-, E_{k'})$.
For any torsion-free finitely presented $\RaG$-module $M$, $\fHom_{\RaG}(M, E_{k'})$ is isogenous over $k$ to a product of powers of the abelian varieties $\fHom_{\FaG}(S_i, E_{k'})$.
\end{ithm}

As an application of this Galois-equivariant functor, we give a simple and new proof of the following old result (see \cite[Proposition 25]{clark}, \cite[Proposition 2.2]{pb}, \cite[Proposition 5.3]{rubin}, \cite[Proposition 2.3b]{kwon} for other proofs of the same or similar results), which is useful in reducing questions about arbitrary CM elliptic curves to those with complex multiplication by a maximal order; this will be used in \cite{vogt} for this very reason.

\begin{icor}\label{maxiso}
Let $k$ be a number field and let $E/k$ be an elliptic curve.  Suppose that $E_{\bar{k}}$ has complex multiplication by an order $\O$ in an imaginary quadratic field $F$.  Then there exists an elliptic curve $E'/k$ and an isogeny 
\[\varphi \colon E' \to E\] 
defined over $k$, such that $E'_{\bar{k}}$ has complex multiplication by the full ring of integers of $F$.
\end{icor}

\subsection*{Acknowledgements} I would like to thank Bjorn Poonen for suggesting that the initial functor could be useful in addressing the problem in Corollary \ref{maxiso} and for helpful discussions.  I also thank Pavel Etingof for answering some questions, and Pete Clark for explaining the history of the result in Corollary \ref{maxiso} and providing references.  This research was partially supported by the National Science Foundation Graduate Research Fellowship Program under Grant No.\ 1122374 as well as Grant DMS-1601946.

\section{Categorical Constructions}\label{cattheory}

For any $\RaG$-module $M$, let $\fHom_{\RaG}(M, E_{k'})$ be the functor from the category of $k$-schemes to the category of sets sending a $k$-scheme $C$ to 
\[ \Hom_{\RaG}(M, \Hom_{k'}(C_{k'}, E_{k'})). \]

\begin{prop}\label{representable}
The functor $\fHom_{\RaG}(M, E_{k'})$ is representable by a commutative proper $k$-group scheme.
\end{prop}
\begin{proof}
Let 
\[\Res_{k'/k}\colon \{\text{comm. proper $k'$-group schemes}\} \to \{\text{comm. proper $k$-group schemes}\}\]
denote Weil restriction of scalars. We first claim that $\fHom_{\RaG}( \RaG, E_{k'}) = \Res_{k'/k} E_{k'}$.
Indeed, by the universal property of the restriction of scalars we have
\begin{align*}
\Hom_k(C, \Res_{k'/k}E_{k'}) &= \Hom_{k'}(C_{k'}, E_{k'}) = \Hom_{\RaG}(\RaG, \Hom_{k'}(C_{k'}, E_{k'})), \\
&= \Hom_k(C, \fHom_{\RaG}(\RaG, E_{k'})), 
\end{align*}
as desired.  Note that the $k$-scheme $\Res_{k'/k}E_{k'}$ has endomorphisms by $\RaG$.

Now let
\begin{equation}\label{Mpres} \RaG^m \to \RaG^n \to M \to 0 \end{equation}
be a finite presentation of $M$ as $\RaG$-modules.  The map $\RaG^m \to \RaG^n$ is given by multiplication on the right by a $m \times n$ matrix $X$.  Multiplication by $X$ on the left also defines a map
\[(\Res_{k'/k}E_{k'})^n \to ( \Res_{k'/k} E_{k'})^m. \]
As commutative proper group schemes over $k$ form an abelian category, the above morphism has a kernel in this category,  
\begin{equation}\label{Apres}0 \to A \to (\Res_{k'/k}E_{k'})^n \to ( \Res_{k'/k} E_{k'})^m. \end{equation}
We claim that $A$ represents the functor $\fHom_{\RaG}(M, E_{k'})$ defined above.  Indeed, we may apply the left-exact functor $\Hom_k(C, -)$ to \eqref{Apres} to obtain
\begin{equation}\label{HomA} 0 \to \Hom_k(C, A ) \to \Hom_{k'}(C_{k'}, E_{k'})^n \to \Hom_{k'}(C_{k'}, E_{k'})^m. \end{equation}
Similarly apply the left-exact functor $\Hom_{\RaG}(-, \Hom_{k'}(C_{k'}, E_{k'}))$ to \eqref{Mpres} to obtain
\begin{equation}\label{HomM}
0 \to \Hom_{\RaG}(M, \Hom_{k'}(C_{k'}, E_{k'})) \to \Hom_{k'}(C_{k'}, E_{k'})^n \to \Hom_{k'}(C_{k'}, E_{k'})^m. 
\end{equation}
As the right maps in both \eqref{HomA} and \eqref{HomM} are induced by multiplication by $X$, the kernels are functorially isomorphic, as desired.
\end{proof}

Recall that we let $F \colonequals R \otimes \qq$.  We now show that we have a compatible functor taking values in the category of abelian varieties over $k$ up to isogeny, as in Theorem \ref{isog}.  For objects $B_1$ and $B_2$ in the isogeny category, denote by $\Hom_{\isog}(B_1, B_2)$ the morphisms in the isogeny category.  

For a finitely presented $\FaG$-module $N$, define the functor $\fHom_{\FaG}(N, E_{k'})$ from the isogeny category of abelian varieties over $k$ to sets by
\[B \mapsto \Hom_{\FaG}(N, \Hom_{\isog}(B_{k'}, E_{k'})). \]
\begin{prop}
The functor $\fHom_{\FaG}(N, E_{k'})$ defined above is represented by an abelian variety in the isogeny category over $k$.
\end{prop}
\begin{proof}
As above, let $B$ be an object of the isogeny category of abelian varieties over $k$.
If $N = M\otimes_{\RaG} \FaG$, we have
\begin{align*}
\Hom_{\FaG}(N, \Hom_{\isog}(B_{k'}, E_{k'})) &= \Hom_{\FaG}(M \otimes \FaG, \Hom_{k'}(B_{k'}, E_{k'})\otimes \qq), \\
&= \Hom_{\RaG}(M, {}_{\RaG} \Hom_{k'}(B_{k'}, E_{k'}) \otimes \qq), \\
&= \Hom_{\RaG}(M,  \Hom_{k'}(B_{k'}, E_{k'}) )\otimes \qq.
\end{align*}
And so this functor agrees with the original $\fHom_{\RaG}(-, E_{k'})$ after composing with the localization map to the isogeny category.  

Therefore the functor $\fHom_{\FaG}(\FaG, E_{k'})$ is represented by $\Res_{k'/k}(E_{k'})$ in the isogeny category.  Furthermore, as the isogeny category is an abelian category, the same proof as in Lemma \ref{representable} shows that for all finitely presented $\FaG$-modules $N$, the functor $\fHom_{\FaG}(N, E_{k'})$ is represented by an object in the isogeny category.
\end{proof}

\begin{lem}\label{semisimple}
Let $F$ be a semisimple $\qq$-algebra.  Then $\FaG$ is semisimple.
\end{lem}
\begin{proof}
Let $V \to W$ be a surjection of finite-dimensional left $\FaG$-modules.  As $F$ is semisimple, we may choose a splitting $\varphi \colon W \to V$ as $F$-modules.  The map $\pi \colon W \to V$ given by
\[\pi(w) = \sum_{g \in G} g \varphi( g^{-1} w ),\]
defines a splitting as $\FaG$-modules.
\end{proof}

Since the endomorphism algebra of an elliptic curve is always semisimple, Lemma \ref{semisimple} combined with the above construction proves Theorem \ref{isog}.

We now show compatibility with base change and restriction of scalars.

\begin{lem}\label{restriction}
Let $M$ be a finitely presented left $R$-module and set $A \colonequals \fHom_R(M, E_{k'})$.  Then we have that
\[\Res_{k'/k} A \simeq \fHom_{\RaG}(\RaG \otimes_R M, E_{k'}). \]
\end{lem}
\begin{proof}
By Yoneda's Lemma, it suffices to show that $\Res_{k'/k} A$ and $\fHom_{\RaG}(\RaG \otimes_R M, E_{k'}$ have the same functors of points.  Let $C$ be a $k$-scheme.  By the universal property of restriction of scalars, we have
\[\Res_{k'/k} A ( C) = A(C_{k'}) = \Hom_R(M, \Hom(C_{k'}, E_{k'})). \]
By the adjunction between restriction and induction \cite[Prop 2.8.3(i)]{benson}, we have
\[\Hom_R(M, \Hom(C_{k'}, E_{k'})) \simeq \Hom_{\RaG}(\RaG \otimes M, \Hom(C_{k'}, E_{k'})), \]
which completes the proof.
\end{proof}

For any ring $R$ and any left $R$-algebra $S$, let $\underline{S}$ denote the $(R, S)$-bimodule $S$ under multiplication by $R$ on the left and $S$ on the right.

\begin{lem}\label{basechange}
Let $M$ be a finitely presented left $\RaG$-module and let $A \colonequals \fHom_{\RaG}(M, E_{k'})$.  Then the base change $A_{k'}$ is isomorphic to
\[\fHom_R({}_R M, E_{k'}),\]
where ${}_R M$ denotes the underlying $R$-module of $M$.
\end{lem}
\begin{proof}
By Yoneda's Lemma, it suffices to show that $A_{k'}$ and $\fHom_R({}_R M, E_{k'})$ have the same functor of points.  Let $D$ be a $k'$-scheme.  Let ${}_k D$ denote the $k$-scheme whose structure morphism is the
composition $D \to \Spec{k'} \to \Spec{k}$.  By the universal property of fiber products,
\[A_{k'}(D) = A({}_kD) = \Hom_{\RaG}\left(M, E_{k'}\big(({}_kD)_{k'}\big) \right). \]
We are thus reduced to showing that
\[\Hom_{\RaG}\left(M, E_{k'}\big(({}_kD)_{k'}\big) \right) = \Hom_R({}_RM, E_{k'}(D)). \]
Furthermore, by the adjunction between restriction and coinduction \cite[Prop 2.8.3(ii)]{benson}, 
\[\Hom_R({}_RM, E_{k'}(D)) = \Hom_{\RaG}(M, \Hom_R(\underline{\RaG}, E_{k'}(D))). \]
It suffices then to show that, as left $\RaG$-modules,
\[E_{k'}\big(({}_kD)_{k'}\big) \simeq \Hom_R(\RaG, E_{k'}(D)). \]
As $k' \otimes_k k' \simeq \prod_{\sigma \in G} k'$, we have that
\[({}_kD)_{k'} \simeq \amalg_{\sigma \in G} D, \]
where $\tau \in G$ maps the $\sigma$th copy of $D$ to the $\tau\sigma$th copy.  Elements of  $E_{k'}\left(\amalg_{\sigma \in G} D \right)$ can be represented as tuples $f = (f_\sigma)_{\sigma}$ where $f_\sigma \in E_{k'}(D)$ and 
\[({}^{\tau}f)_\sigma = \tau f_{\tau^{-1}\sigma}. \]
Similarly elements of $\Hom_R(\RaG, E_{k'}(D))$ are tuples $g = (g_\sigma)_\sigma$ where $g \in E_{k'}(D)$ and $({}^\tau g)_\sigma = g_{\sigma \tau}$.  The correspondence
\[ g_\sigma = \sigma f_{\sigma^{-1}}, \]
shows that these $\RaG$-modules are isomorphic.
\end{proof}


Note that any finitely presented $R$-module with a semilinear $G$-action is also a finitely presented $\RaG$-module, as we now explain.  As $R$ is noetherian, it suffices to show that any $\RaG$-module $M$, which is finitely generated as an $R$-module, admits a $G$-equivariant surjection from $\RaG^n$ for some $n$.  If $m_1, \cdots, m_r$ are $R$-module generators of $M$, then the map $\RaG^r \to M$ sending $e_i \mapsto m_i$ has the desired properties.  

We have the following nice properties of the functor $\fHom_{R\langle G \rangle}(-, E_{k'})$.

\begin{prop}\label{nice_prop}
Let $E/k$ be an elliptic curve and $k'/k$ a finite Galois extension with Galois group $G$.  Let $R \colonequals \End_{k'}E$ and let $M$ be a finitely presented $\RaG$-module and let $A \colonequals \fHom_{\RaG}(M, E_{k'})$.  
\begin{enumerate}
\item $\fHom_{\RaG}(-, E_{k'})$ is exact.
\item $A$ is a commutative proper group scheme over $k$ of dimension $\rk_R(M)$.
\item If $M$ is torsion-free as an $R$-module, then $A$ is an abelian variety over $k$ such that $A_{k'}$ is isogenous to a power of $E_{k'}$.
\item $\fHom_{\RaG}(R, E_{k'}) = E$.
\end{enumerate}
\end{prop}
\begin{proof}
Parts (1) -- (3) follow from the corresponding properties after base-extension to $k'$ \cite[Thm 4.4]{bjorn}.  For part (4), we have
\begin{align*}
\Hom_k(C, \fHom_{\RaG}(R, E_{k'})) &= \Hom_{\RaG}(R, \Hom_{k'}(C_{k'}, E_{k'})) \\
&= \Hom_{k'}(C_{k'}, E_{k'})^G \\
&= \Hom_k(C, E). \qedhere
\end{align*}
\end{proof}

\section{Examples and Applications}

As an example of Theorem \ref{isog}, we have the following in the special case $k'/k$ is a separable quadratic extension.
\begin{prop}
Let $k'/k$ be a separable quadratic extension and set $G \colonequals \Gal(k'/k)$ with nontrivial element $\sigma$.  Let ${}^\sigma E$ denote the corresponding quadratic twist of $E$.  If $M$ is a torsion-free $\RaG$-module, then $\fHom_{\RaG}(M, E_{k'})$,
is isogenous to $E^r \times\left( {}^\sigma E\right)^{r'}$ for some $r,r' \geq 0$.
\end{prop}
\begin{proof}
As above, we denote the $G$-action on $F$ by $\sigma(x) =  {}^\sigma x$.  By Theorem \ref{isog}, it suffices to decompose $\FaG$ into simple modules.
Let ${}^\sigma F$ denote $F$ endowed with the $G$-action $\sigma (x) = - {}^\sigma x$ for $x\in F$.  Then we have an isomorphism
\[F \oplus {}^\sigma F \xrightarrow{\varphi} \FaG, \]
defined by $\varphi(a,b) = a(\sigma + 1) +  b(\sigma - 1)$.  By comparing functors of points, we have $\fHom_{\FaG}(F, E_{k'}) \simeq E$ and $\fHom_{\FaG}({}^\sigma F, E_{k'}) \simeq {}^\sigma E$ in the isogeny category.\end{proof}

\begin{example}\label{main_ex}
Consider the special case that $k$ is a number field, and $E$ has complex multiplication by an order $\O$ in the imaginary quadratic field $F= \Frac \O$, which is defined over the quadratic extension $k' = kF/k$ with Galois group $G = \{1, \sigma\}$.  Let $M$ be a finitely presented $\O \langle G \rangle$-module that is torsion-free as an $\O$-module.  Then as multiplication by a totally imaginary element of $\O$ defines an isomorphism ${}^\sigma F \simeq F$ of $\FaG$-modules, we have that $\fHom_{\O \langle G \rangle}(M, E_{k'})$ is an abelian variety defined over $k$, which is isogenous over $k$ to a power of $E$.
\end{example}

%

Continuing in the setup of the previous example, we conclude by proving Corollary \ref{maxiso}.

\begin{proof}[Proof of Corollary \ref{maxiso}]
This follows from Example \ref{main_ex} as the full ring of integers $\O_F$ of $F$ is a finitely presented $\O \langle G \rangle$-module, which is torsion-free of rank $1$ over $\O$.
The full endomorphism ring of $E$ is defined over (the at most quadratic extension) $k'$, so we have that $\O = \End E_{k'}$.  This inherits a (possibly trivial) action of $G$.  
We therefore have the following exact sequence of $\O\langle G \rangle$-modules:
\[0 \to \O \to \O_F \to \O_F/\O \to 0, \]
where the action of $G$ is induced from the action on $F$.  Applying the functor $\fHom_{\O\langle G \rangle}(-, E_{k'})$ we obtain an exact sequence
\[0 \to \fHom_{\O\langle G \rangle}(\O_F/\O, E_{k'}) \to E' \to E \to 0, \]
where $E'$ is again an elliptic curve over $k$ and the right map is an isogeny to $E$.  By functoriality, $E'_{k'}$ has an action of $\O_F$, as desired.
\end{proof}

\bibliographystyle{plain}
\bibliography{ehom}

\end{document}